\documentclass[a4paper,10pt]{amsart}

\usepackage[utf8]{inputenc} 
\usepackage{amsmath}
\usepackage{amsbsy}
\usepackage{amssymb}
\usepackage{amscd,amsthm}
\usepackage{verbatim}
\usepackage[english]{babel}
\usepackage{dsfont}
\usepackage{anysize}
\usepackage{hyperref}

\newcommand{\e}{\varepsilon}

\newcommand{\N}{\mathds{N}}

\newcommand{\q}{\quad}

\newtheorem{thm}{Theorem}[section]
\newtheorem{lem}[thm]{Lemma}

\newtheorem{prop}[thm]{Proposition}
\newtheorem*{prim}{Prime $k$-tuples Conjecture}

\theoremstyle{definition}
\newtheorem*{defi}{Definition}

\theoremstyle{remark}
\newtheorem*{rema}{Remark}

\title{Some Results on a Conjecture of Hardy and Littlewood}
\author{Christian Axler}
\address{Institute of Mathematics\\ Heinrich-Heine University Düsseldorf\\
40225 Düsseldorf, Germany}
\email{christian.axler@hhu.de}
\date{\today}

\begin{document}

\begin{abstract}
Let $m$ and $n$ be positive integers with $m,n \geq 2$. The second Hardy-Littlewood conjecture states that the number of primes in the interval $(m,m+n]$ is 
always less than or equal to the number of primes in the interval $[2,n]$. Based on new explicit estimates for the prime counting function $\pi(x)$, we 
give some new ranges in which this conjecture holds.
\end{abstract}

\maketitle

\section{Introduction}

The prime counting function $\pi(x)$ denotes the number of primes less or equal to $x$. In 1872, Lionnet \cite{lionnet} raised the question whether the 
inequality
\begin{equation}
\pi(2n) - \pi(n) \leq \pi(n) \tag{1.1} \label{1.1}
\end{equation}
holds for every integer $n \geq 2$. This means that for each integer $n \geq 2$ the interval $(n,2n]$ contains at most as many prime numbers as the interval 
$[2,n]$. A first progress concerning this question was done by Landau \cite[p.\:215--216]{landau}. He used the Prime Number Theorem, i.e.
\begin{displaymath}
\pi(x) = \frac{x}{\log x} + O \left( \frac{x}{\log^2x} \right)
\end{displaymath}
as $x \to \infty$, to show that \eqref{1.1} holds for every sufficiently large positive integer $n$. In 1923, Hardy and Littlewood \cite{hardy} conjectured that
\begin{displaymath}
\limsup_{n \to \infty} (\pi(x+n) - \pi(n)) \leq \pi(x)
\end{displaymath}
for every $x \geq 2$. From here it has been derived the well known conjecture (in the following denoted by HLC) that
\begin{equation}
\pi(m+n) \leq \pi(m) + \pi(n) \q\q \forall \, m,n \in \N \setminus \{ 1 \}. \tag{1.2} \label{1.2}
\end{equation}
Clearly, the HLC is a generalization of Lionnet's question \eqref{1.1}. Although the HLC could neither be proven nor disproved in general so far, some 
special cases can be shown. As a consequence of their explicit estimates for $\pi(x)$, Rosser and Schoenfeld \cite{rosser1975} stated without proof that
\begin{equation}
\pi(2x) - \pi(x) \leq \pi(x) \tag{1.3} \label{1.3}
\end{equation}
for every real $x \geq 3$. A detailed proof was finally given by Kopetzky and Schwarz \cite{kopetzky}. If we combine \eqref{1.3} with $\pi(4) - \pi(2) = 
\pi(2)$, it turns out that Lionnet's inequality \eqref{1.1} indeed holds for every integer $n \geq 2$.
Erdös \cite{erdos} reported that Ungar verified the HLC for every pair of 
integers $(m,n)$ satisfying $2 \leq \min(m,n) \leq 41$. One year later, Schinzel and Sierpi\'nski \cite{schinzel} could show that the inequality is fulfilled 
for every pair of integers $(m,n)$ with $2 \leq \min(m,n) \leq 132$. In a later paper, Schinzel \cite{schinzel2} extended this range to $2 \leq \min(m,n) \leq 
146$. The current best result in this direction was given by Gordan and Rodemich \cite{rodemich}. They found that the HLC is fullfilled for every pair of 
integers $(m,n)$ satisfying
\begin{equation}
2 \leq \min(m,n) \leq 1731. \tag{1.4} \label{1.4}
\end{equation}
The next result is due to Panaitopol \cite[Theorem 1]{pana1}. He showed that the HLC is true for every pair of integers $(m,n)$ satisfying $m, n \geq 2$ and
\begin{displaymath}
\frac{m}{29} \leq n \leq m.
\end{displaymath}
In \cite[Proposition 3]{dusart2}, Dusart improved the result of Panaitopol by showing that the HLC is true for every pair of positive real numbers $(x,y)$ 
satisfying $x, y 
\geq 3$ and
\begin{equation}
\frac{x}{109} \leq y \leq x. \tag{1.5} \label{1.5}
\end{equation}
Using explicit estimates for the prime counting function $\pi(x)$, we find the following improvement.

\begin{thm} \label{thm101}
Let $m$ and $n$ be integers satisfying $m, n \geq 2$ and $m/1950 \leq n \leq m$. Then we have
\begin{displaymath}
\pi(m+n) \leq \pi(m) + \pi(n).
\end{displaymath}
\end{thm}

In 1975, Udrescu \cite{udrescu} found the following generalization. Under the assumption that $n$ satisfies $\e m \leq n \leq m$, where $\e$ is a real number 
with $0 < \e \leq 1$, he showed that the HLC holds for every sufficiently large positive integer $m$. Dusart \cite{dusart2} showed that Udrescu's result holds 
for every integer $m \geq e^{3.1/\log(1+\e)}$. We give the following improvement.

\begin{thm} \label{thm102}
Udrescu's result holds for every integer
\begin{displaymath}
m \geq e^{\sqrt{0.3426/\log(1+\e)}}.
\end{displaymath}
\end{thm}

In \cite{pana2}, Panaitopol \cite{pana2} used explicit estimates for the prime counting function $\pi(x)$ to get that the HLC is true for all positive integers 
$m,n \geq 2$ with
\begin{equation}
\pi(m) \leq n \leq m. \tag{1.6} \label{1.6}
\end{equation}
Since $\pi(x) \sim x/\log x$ as $x \to \infty$, the last result yields an improvement of Theorem \ref{thm101} for all sufficiently large values of $m$. In this 
paper, we find the following refinement of \eqref{1.6}.

\begin{thm} \label{thm103}
Let $c_0 = 0.70881678090424862707121$. Then we have $\pi(m+n) \leq \pi(m) + \pi(n)$ for all integers $m \geq n \geq 2$ with $n \geq c_0m/\log^2 m$.
\end{thm}

In the case where $m+n \leq 10^{20}$, we can use some recent results concerning the distance of $\pi(x)$ and the \emph{logarithmic integral} $\text{li}(x)$, 
which is defined for every real $x > 1$ as
\begin{displaymath}
\text{li}(x) = \int_0^x \frac{\text{d}t}{\log t} = \lim_{\varepsilon \to 0+} \left \{ \int_{0}^{1-\varepsilon}{\frac{\text{d}t}{\log t}} + 
\int_{1+\varepsilon}^{x}{\frac{\text{d}t}{\log t}} \right \},
\end{displaymath}
to get the following result.

\begin{thm} \label{thm104}
Let $c_1 = 2(1-\log 2) = 0.6137 \ldots$. Then we have $\pi(m+n) \leq \pi(m) + \pi(n)$ for all integers $m \geq n \geq 2$ satisfying $m + n \leq 
10^{20}$ and
\begin{displaymath}
n \geq 2\sqrt{m} \left( 1 - \frac{2c_1}{\log m + c_1} \right).
\end{displaymath}
%
%
\end{thm}

Finally, we find the following result which depends on the correctness of the Riemann hypothesis.

\begin{thm} \label{thm105}
Let $c_2 = 1/(4\pi)$. If the Riemann hypothesis is true, then $\pi(m+n) \leq \pi(m) + \pi(n)$ for all integers $m \geq n \geq 2$ satisfying $n \geq c_2 
\sqrt{m}\log m \log(m\log^8m)$.
\end{thm}

%

\section{On a result of Segal}

In 1962, Segal \cite[Theorem I]{segal} obtained the following inequality condition involving only prime numbers which is equivalent to the HLC. Here, as 
usual, $p_r$ denotes the $r$th prime number.

\begin{lem}[Segal] \label{lem201}
The \emph{HLC} is true if and only if
\begin{equation}
p_k \geq p_{k-q} + p_{q+1} - 1 \tag{2.1} \label{2.1}
\end{equation}
for all integers $k,q$ satisfying $k \geq 3$ and $1 \leq q \leq (k-1)/2$.
\end{lem}

Then, Segal \cite[Theorem II]{segal} used this equivalence to get the following result.

\begin{lem}[Segal] \label{lem202}
If the \emph{HLC} is false for some positive integer $m+n$, then the smallest such value of $m+n$ is the smallest value of $p_k$ for which \eqref{2.1} is false.
\end{lem}

He used a computer to see that the inequality \eqref{2.1} holds for every positive integer $k \leq 9679$; i.e. for every prime number $p_k \leq 101\,081$. Now 
it follows from Lemma \ref{lem202} that the HLC holds for all integers $m,n \geq 2$ with $m+n \leq 101\,081$. In 2001, Panaitopol \cite{pana1} improved Lemma 
\ref{lem201} by showing the following

\begin{lem}[Panaitopol] \label{lem203}
The \emph{HLC} is true if and only if the inequality \eqref{2.1} holds for all integers $k,q$ satisfying $k \geq 9680$ and $34 \leq q \leq (k-1)/27$.
\end{lem}

Using Lemmata \ref{lem202} and \ref{lem203} and a computer, Panaitopol \cite{pana1} found that the HLC is true for all integers $m,n \geq 2$ with $m+n 
\leq 3\,497\,861 = p_{250000}$. Extending this computation, we get the following

\begin{prop} \label{prop204}
Let $N_0 = 1.7 \times 10^9$. Then the \emph{HLC} holds for all integers $m,n \geq 2$ satisfying $m+n \leq 39\,708\,229\,123 = p_{N_0}$.
\end{prop}

\section{A Proof of Theorem \ref{thm101}}

First, we set
\begin{equation}
f_c(t) = \frac{t}{\log t - 1 - c/\log t}. \tag{3.1} \label{3.1}
\end{equation}
By \cite[Corollary 3.5]{axler16}, we have $\pi(t) \geq f_1(t)$ for every $t \geq 468049$. Let $b$ a real number with $b \in (1,2)$ and let $B$ a positive real 
number so that $\pi(t) \leq f_b(t)$ for every $x \geq B$. Further, let $r$ and $s$ be positive real numbers with $r \geq s \geq 1$. We set
\begin{equation}
\lambda_b(r,s) = \frac{(b-1)(r+1) - \log(s+1)\log s}{2r \log(1+\frac{1}{r}) + 2\log(s+1)} + \frac{1}{2} \left( \log r - \log \left( 1 + \frac{1}{r} \right) 
\right) \tag{3.2} \label{3.2}
\end{equation}
and
\begin{displaymath}
\eta_b(r,s) = \frac{r\log r - \log s - (1+\log(s+1)\log s)\log(1+\frac{1}{r}) - br\log s}{r \log(1+\frac{1}{r}) + \log(s+1)} + \log \left( 1 + \frac{1}{r} 
\right) \log s.
\end{displaymath}
Then we get the following result.

\begin{prop} \label{prop301}
Let $r$ and $s$ be real numbers with $r \geq s \geq 1$. Let $\chi_b(r,s) = \lambda_b(r,s)^2 + \eta_b(r,s)$ and
\begin{displaymath}
\varphi_b(r,s) = \frac{\emph{sgn}(\chi_b(r,s)) + 1}{2} \cdot \chi_b(r,s).
\end{displaymath}
Then we have $\pi(x+y) \leq \pi(x) + \pi(y)$ for every pair of real numbers $(x,y)$ satisfying $x \geq y \geq 3$,
\begin{equation}
x \geq \max \left\{ \exp(\lambda_b(r,s) + \sqrt{\varphi_b(r,s)}), 468049r, \frac{B}{1+1/r} \right\}, \tag{3.3} \label{3.3}
\end{equation}
and $x/r \leq y \leq x/s$.
\end{prop}

\begin{proof}
By \eqref{1.5}, we can assume that $r \geq s \geq 109$. Let $h(x,y) = \pi(x) + \pi(y) - \pi(x+y)$. We need to show that $h(x,y) \geq 0$. First, we note that
\begin{equation}
\log \left( 1 + \frac{x}{y} \right) - \frac{b}{\log(x+y)} + \frac{1}{\log y} \geq \log 110 - \frac{2}{\log 468049} \geq 0. \tag{3.4} \label{3.4}
\end{equation}
Since $\log(x/y) \geq \log s$, we have
\begin{equation}
\frac{1}{\log y - 1 - \frac{1}{\log y}}
\geq \frac{1 + \frac{\log s}{\log y}}{\log x - 1 - \frac{1}{\log x}}. \tag{3.5} \label{3.5}
\end{equation}
From \eqref{3.3}, it follows that $(\log x - \lambda_b(r,s))^2 \geq \lambda_b(r,s)^2 + \eta_b(r,s)$. Substituting the definition of $\eta_b(r,s)$ into the last 
inequality, we see that
\begin{align*}
\log^2x & - 2\lambda_b(r,s)\log x - \log \left( 1 + \frac{1}{r} \right) \log r \geq \frac{r\log r - \log s - (1+\log(s+1)\log s)\log(1+\frac{1}{r}) - br\log 
s}{r \log(1+\frac{1}{r}) + \log(s+1)}. 
\end{align*}
Let $\kappa(r,s) = r \log(1+1/r) + \log(s+1)$. Then we can use \eqref{3.2} to get that the last inequality is equivalent to
\begin{align*}
\kappa(r,s)\log \left(x + \frac{x}{r} \right)\log \frac{x}{r} & - b(r+1)\log \frac{x}{s} + r \log \frac{x}{r} \\
& \geq (b-1)\log s - (1+\log(s+1)\log s)\log \left( x+\frac{x}{r} \right).
\end{align*}
Since $x/r \leq y \leq x/s$, we get
\begin{displaymath}
\frac{\kappa(r,s)}{r} - \frac{b(1+\frac{1}{r})}{\log(x+y)} + \frac{1}{\log(x+y)} \geq \frac{(b-1)\log s}{r\log(x+y)\log y} - \frac{1}{r\log y} - 
\frac{\log(s+1)\log s}{r\log y}.
\end{displaymath}
Hence
\begin{displaymath}
\frac{\kappa(r,s)}{r} - \frac{b(1+\frac{1}{r})}{\log(x+y)} + \frac{1}{\log x} \geq \frac{b\log s}{r\log(x+y)\log y} - \frac{\log s}{r\log^2 y} - \frac{1}{r\log 
y} - \frac{\log(s+1)\log s}{r\log y}.
\end{displaymath}
Now we substitute the definitoin of $\kappa(r,s)$ to obtain the inequality
\begin{displaymath}
\log \left( 1+\frac{1}{r} \right) - \frac{b}{\log(x+y)} + \frac{1}{\log x} + \frac{1}{r} \left( \log \left( 1 + \frac{x}{y} \right) - \frac{b}{\log(x+y)} + 
\frac{1}{\log y} \right) \left( 1 + \frac{\log s}{\log y} \right) \geq 0.
\end{displaymath}
Therefore,
\begin{equation}
\log \left( 1+\frac{y}{x} \right) - \frac{b}{\log(x+y)} + \frac{1}{\log x} + \frac{1}{r} \left( \log(s+1) - \frac{b}{\log(x+y)} + \frac{1}{\log y} \right) 
\left( 1 + \frac{\log s}{\log y} \right) \geq 0. \tag{3.6} \label{3.6}
\end{equation}
Next, we note that $y \geq x/r \geq 468049$ and $x+y \geq x(1+1/r) \geq B$. Hence $h(x,y) \geq f_1(x) + f_1(y) - f_b(x+y)$, where $f_c(t)$ is defined as in 
\eqref{3.1}. Setting $g_c(t) = \log t - 1 - c/\log t$, we see that
\begin{displaymath}
h(x,y) \geq x \left( \frac{\log(1 + \frac{y}{x}) - \frac{b}{\log(x+y)} + \frac{1}{\log x}}{g_1(x)g_b(x+y)} \right) + y \left( \frac{\log(1 + \frac{x}{y}) - 
\frac{b}{\log(x+y)} + \frac{1}{\log y}}{g_1(y)g_b(x+y)} \right).
\end{displaymath}
Now we can use \eqref{3.4} and \eqref{3.5} to get the inequality
\begin{displaymath}
h(x,y) \geq x \left( \frac{\log(1 + \frac{y}{x}) - \frac{b}{\log(x+y)} + \frac{1}{\log x} + \frac{1}{r}(\log(1 + \frac{x}{y}) - \frac{b}{\log(x+y)} + 
\frac{1}{\log y})(1 + \frac{\log s}{\log y})}{g_1(x)g_b(x+y)} \right).
\end{displaymath}
Finally it suffices to apply the inequality \eqref{3.6}.
\end{proof}

Now we use Propositions \ref{prop204} and \ref{prop301} to give the following proof of Theorem \ref{thm101}.

\begin{proof}[Proof of Theorem \ref{thm101}]
We set $b = 1.15$. By \cite[Corollary 1]{axler2017}, we can choose $B = 38\,284\,442\,297$. In addition, we substitute the following explicit values for $r$ 
and $s$ into Proposition \ref{prop301} to get $\pi(x+y) \leq \pi(x) + \pi(y)$ for every $x \geq x_0$ and $x/r \leq y \leq x/s$, where $x_0$ is equal to the 
least integer greater than or equal to the right-hand side of \eqref{3.3}:

\begin{center}
\begin{tabular}{|l||c|c|c|c|c|}
\hline
$r$  \rule{0mm}{4mm} &              $1950$ &         $1949.9652$ &         $1949.8838$ &         $1949.6933$ &         $1949.2476$ \\ \hline
$s$  \rule{0mm}{4mm} &         $1949.9652$ &         $1949.8838$ &         $1949.6933$ &         $1949.2476$ &         $1948.2049$ \\ \hline
$x_0$\rule{0mm}{4mm} & $38\,284\,409\,814$ & $38\,284\,393\,330$ & $38\,284\,407\,670$ & $38\,284\,394\,575$ & $38\,284\,419\,151$ \\ \hline \hline
$r$\rule{0mm}{4mm}   &         $1948.2049$ &         $1945.7667$ &         $1940.0707$ &         $1926.7942$ &         $1896.0125$ \\ \hline
$s$\rule{0mm}{4mm}   &         $1945.7667$ &         $1940.0707$ &         $1926.7942$ &         $1896.0125$ &         $1825.5323$ \\ \hline 
$x_0$\rule{0mm}{4mm} & $38\,284\,398\,522$ & $38\,284\,417\,850$ & $38\,284\,399\,116$ & $38\,284\,426\,596$ & $38\,284\,405\,535$ \\ \hline \hline
$r$\rule{0mm}{4mm}   &         $1825.5323$ &         $1668.8817$ &         $1344.8932$ &         $ 785.8821$ &         $ 189.9788$ \\ \hline
$s$\rule{0mm}{4mm}   &         $1668.8817$ &         $1344.8932$ &         $ 785.8821$ &         $ 189.9788$ &         $ 109     $ \\ \hline
$x_0$\rule{0mm}{4mm} & $38\,284\,440\,640$ & $38\,284\,412\,784$ & $38\,284\,406\,728$ & $38\,284\,305\,355$ & $38\,083\,977\,941$ \\ \hline
\end{tabular}
\end{center}
In particular, we see that $\pi(m+n) \leq \pi(m) + \pi(n)$ for every $m \geq 38\,284\,440\,640$ and $m/1950 \leq n \leq m/109$. If $m \leq 38\,284\,440\,640$ 
and $m/1950 \leq n \leq m/109$, we get $m+n \leq (1+1/109)m \leq 39\,708\,229\,123$ and the result follows from Proposition \ref{prop204}. The remaining case 
where $m,n \geq 2$ and $m/109 \leq n \leq m$ is a direct consequence of \eqref{1.4} and \eqref{1.5}.
\end{proof}

\section{A Proof of Theorem \ref{thm102}}

In this section, we use explicit estimates for the prime counting function $\pi(x)$ to give the following proof of Theorem \ref{thm102}.

\begin{proof}[Proof of Theorem \ref{thm102}]
If $\e \in [1/1950, 1]$, the result follows from Theorem \ref{thm101}. So, let $\e \in (0, 1/1950)$ and let $m,n \geq 2$ be integers with $\e m \leq n \leq m$ 
and $m \geq e^{\sqrt{0.3426/\log(1+\e)}}$. Then $m \geq 168\,527\,259\,431$. Hence,
\begin{displaymath}
\log(1+\e) \geq \frac{0.3426}{\log^2m} \geq \frac{0.3}{\log^2m} + \frac{1.1}{\log^3m}.
\end{displaymath}
Hence
\begin{displaymath}
\log(m+n) - 1 - \frac{1}{\log(m+n)} - \frac{3.15}{\log^2 m} - \frac{14.25}{\log^3 m} \geq \log m - 1 - \frac{1}{\log m} - \frac{2.85}{\log^2 m} - 
\frac{13.15}{\log^3 m}.
\end{displaymath}
Now we can use \cite[Corollary 3]{axler2017} to see that
\begin{equation}
\frac{m}{\log(m+n) - 1 - \frac{1}{\log(m+n)} - \frac{3.15}{\log^2(m+n)} - \frac{14.25}{\log^3(m+n)}} \leq \pi(m). \tag{4.1} \label{4.1}
\end{equation}
Since $\log 2 > 3.15/\log^2m + 14.25/\log^3m$, we get
\begin{displaymath}
\log(m+n) - 1 - \frac{1}{\log(m+n)} - \frac{3.15}{\log^2(m+n)} - \frac{14.25}{\log^3(m+n)} \geq \log n - 1 - \frac{1}{\log n}. \tag{4.2} \label{4.2}
\end{displaymath}
Note that the function $f(x) = xe^{\sqrt{0.3426/\log(1+x)}}$ is decreasing on the interval $(0, 1/1950)$. Hence we get $n \geq \e m \geq f(1/1950) \geq 
86\,424\,235$. If we combine the inequality \eqref{4.2} with Corollary 3.5 of \cite{axler16}, it turns out that the inequality
\begin{equation}
\frac{n}{\log(m+n) - 1 - \frac{1}{\log(m+n)} - \frac{3.15}{\log^2(m+n)} - \frac{14.25}{\log^3(m+n)}} \leq \pi(n) \tag{4.3} \label{4.3}
\end{equation}
holds. By \cite[Corollary 1]{axler2017}, we have
\begin{displaymath}
\pi(m+n) \leq \frac{m+n}{\log(m+n) - 1 - \frac{1}{\log(m+n)} - \frac{3.15}{\log^2(m+n)} - \frac{14.25}{\log^3(m+n)}}
\end{displaymath}
and it suffices to apply \eqref{4.1} and \eqref{4.3}.
\end{proof}

\section{A Proof of Theorem \ref{thm103}}

Let $k$ be a positive integer and $\e$ be positive real number. By Panaitopol \cite{pan3}, there exist positive real numbers $a_1, \ldots, a_k$ and two 
positive real numbers $\alpha_k$ and $\beta_k = \beta_k(\e)$ so that
\begin{equation}
\pi(x) \geq \frac{x}{\log x - 1 - \sum_{j=1}^k \frac{a_j}{\log^jx}} \q\q (x \geq \alpha_k) \tag{5.1} \label{5.1}
\end{equation}
and
\begin{equation}
\pi(x) \leq \frac{x}{\log x - 1 - \sum_{j=1}^k \frac{a_j}{\log^jx} - \frac{\e}{\log^kx}} \q\q (x \geq \beta_k). \tag{5.2} \label{5.2}
\end{equation}
Further, let $\gamma_k = \gamma_k(\e)$ be the smallest positive integer so that
\begin{displaymath}
\log 2 \geq \frac{\e}{\log^kx} + \sum_{j=1}^k \frac{a_j}{\log^jx}
\end{displaymath}
for every $x \geq \gamma_k$. Then we obtain the following result.

\begin{prop} \label{prop401}
Let $k$ be a positive integer and $\e, c$ be positive real numbers with $c > \e$. Then $\pi(x+y) \leq \pi(x) + \pi(y)$ for all real numbers $x,y \geq 2$ with 
$x 
\geq \max \{\alpha_k, \beta_k, \gamma_k, \exp(\sqrt[k]{c^2/(2(c - \e))}) \}$ and
\begin{displaymath}
\max \left\{ 5393, \frac{cx}{\log^k x} \right\} \leq y \leq x.
\end{displaymath}
\end{prop}

\begin{proof}
Since $x \geq \exp(\sqrt[k]{c^2/(2(c - \e))})$, we have
\begin{displaymath}
\frac{c}{\log^kx} - \frac{c^2}{2\log^{2k}x} \geq \frac{\e}{\log^kx}.
\end{displaymath}
Using the inequality $\log(1+t) \geq t - t^2/2$, which holds for every $t \geq 0$, we see that
\begin{displaymath}
\log (x+y) - \log x \geq \log \left( 1+ \frac{c}{\log^kx} \right) \geq \frac{\e}{\log^kx}.
\end{displaymath}
If we combine the last inequality with \eqref{5.1}, it turns out that
\begin{equation}
\frac{x}{\log (x+y) - 1 - \sum_{j=1}^k \frac{a_j}{\log^j(x+y)} - \frac{\e}{\log^k(x+y)}} \leq \frac{x}{\log x - 1 - \sum_{j=1}^k \frac{a_j}{\log^j x}} \leq 
\pi(x). \tag{5.3} \label{5.3}
\end{equation}
On the other hand, we have $y \leq x$ and $x \geq \gamma_k$. Hence
\begin{equation}
\frac{y}{\log (x+y) - 1 - \sum_{j=1}^k \frac{a_j}{\log^j(x+y)} - \frac{\e}{\log^k(x+y)}} \leq \frac{y}{\log y - 1}. \tag{5.4} \label{5.4}
\end{equation}
By Dusart \cite[p.\:55]{dusart}, we have $\pi(t) \geq t/(\log t - 1)$ for every $t \geq 5393$. Applying this to \eqref{5.4}, we get
\begin{equation}
\frac{y}{\log (x+y) - 1 - \sum_{j=1}^k \frac{a_j}{\log^j(x+y)} - \frac{\e}{\log^k(x+y)}} \leq \pi(y). \tag{5.5} \label{5.5}
\end{equation}
By \eqref{5.2}, we have
\begin{displaymath}
\pi(x+y) \leq \frac{x+y}{\log (x+y) - 1 - \sum_{j=1}^k \frac{a_j}{\log^j(x+y)} - \frac{\e}{\log^k(x+y)}}
\end{displaymath}
and it suffices to apply \eqref{5.3} and \eqref{5.5}.
\end{proof}

\begin{proof}[Proof of Theorem \ref{thm103}]
Let $k=2$. We set $a_1 = 1$ and $a_2 = 2.85$. By \cite[Corollary 3]{axler2017}, we can choose $\alpha_2 = 38\,099\,531$. Further, we set $\e = 
0.70863503301170907614119$. Then we can use \cite[Theorem 2]{axler2017} to see that \eqref{5.2} holds for every $x \geq \beta_2 = 14\,000\,264\,036\,190\,262$. 
A simple calculation shows that $\gamma_2 = 23$. Now let $c = c_0$. Substituting these values into Proposition \ref{prop301}, it turns out that the inequality 
$\pi(m+n) \leq \pi(m) + \pi(n)$ holds for all integers $m \geq n \geq 2$ satisfying $m \geq 14\,000\,264\,036\,190\,263$ and $n \geq cm/\log^2 m$. If $m \leq 
14\,000\,264\,036\,190\,262$, the claim follows from Theorem \ref{thm101}.
\end{proof}

\section{A Proof of Theorem \ref{thm104}}

First, we note some results of Dusart \cite{dusart2018} concerning the distance of $\pi(x)$ and $\text{li}(x)$.

\begin{prop}[Dusart]
For every real $x$ with $2 \leq x \leq 10^{20}$, we have
\begin{equation}
\pi(x) \leq \emph{li}(x), \tag{6.1}  \label{6.1}
\end{equation}
and for every real $x$ satisfying $1\,090\,877\leq x \leq 10^{20}$, we have
\begin{equation}
\emph{li}(x) - \frac{2\sqrt{x}}{\log x} \leq \pi(x). \tag{6.2}  \label{6.2}
\end{equation}
\end{prop}

\begin{proof}
See \cite[Lemma 2.2]{dusart2018}.
\end{proof}

Now we use this result to find the following proof of Theorem \ref{thm104}.


\begin{proof}[Proof of Theorem \ref{thm104}]
By Theorems \ref{thm101} and \ref{thm103}, it suffices to consider the case where $n$ satisfies
\begin{displaymath}
2 \sqrt{m} \leq n \leq m \times \min \left \{ \frac{1}{1950}, \frac{c_0}{\log^2 m} \right \},
\end{displaymath}
where $c_0$ is given as in Theorem \ref{thm103}. If $m \leq 39\,687\,876\,365$, we get $m+n \leq (1+/1950)m \leq 39\,708\,229\,123$ and the result follows from 
Proposition \ref{prop204}. So we can assume that $m \geq 39\,687\,876\,366$. Using \eqref{6.1}, we see that $\pi(m+n) \leq \text{li}(m+n)$. Now we can use the 
mean value theorem to see that $\pi(m+n) \leq \text{li}(m) + n/\log m$.
Applying \eqref{6.2} to this inequality, we get
\begin{displaymath}
\pi(m+n) \leq \pi(m) + \frac{2\sqrt{m}}{\log m} + \frac{n}{\log m},
\end{displaymath}
which is equivalent to
\begin{equation}
\pi(m+n) \leq \pi(m) + \frac{2\sqrt{m}}{\log m} + \frac{n}{\log n - 1} - \frac{n(\log(m/n) + 1)}{\log m (\log n - 1)}. \tag{6.3} \label{6.3}
\end{equation}
Since $m \geq 39\,687\,876\,366$, we have $n \geq 887\,293$.So we can apply the inequality including $\pi(x)$ given in \cite[p.\: 55]{dusart} to \eqref{6.3} 
and get
\begin{equation}
\pi(m+n) \leq \pi(m) + \pi(n) + \frac{2\sqrt{m}}{\log m} - \frac{n(\log(m/n) + 1)}{\log m (\log n - 1)}. \tag{6.4} \label{6.4}
\end{equation}
In order to prove the theorem, we consider the following three cases.

\textit{Case} 1. $\sqrt{m}\log m/\log \log m \leq n \leq c_0m/\log^2 m$. \newline
In this first case, the inequality \eqref{6.4} implies that
\begin{equation}
\pi(m+n) \leq \pi(m) + \pi(n) + \frac{2\sqrt{m}}{\log m} - \frac{n(2 \log \log m - c_3)}{\log m (\log m - 2\log \log m + c_3)}, \tag{6.5} 
\label{6.5}
\end{equation}
where $c_3 = \log(c_0) - 1 = -1.17409\ldots$. The assumption $n \geq \sqrt{m} \log m/\log \log m$ implies that
\begin{displaymath}
\frac{n(2 \log \log m - c_3)}{\log m (\log m - 2\log \log m + c_3)} \geq \frac{2\sqrt{m}}{\log m}.
\end{displaymath}
Applying this to \eqref{6.5}, we obtain the inequality $\pi(m+n) \leq \pi(m) + \pi(n)$.

\textit{Case} 2. $2 \sqrt{m} (1 + 4 \log \log m/\log m) \leq n \leq \sqrt{m}\log m/\log \log m$. \newline
Here, the inequality \eqref{6.4} implies that
\begin{equation}
\pi(m+n) \leq \pi(m) + \pi(n) + \frac{2\sqrt{m}}{\log m} - \frac{n(\log(\sqrt{m}\log \log m/\log m) + 1)}{\log m (\log(\sqrt{m}\log m/\log \log m) - 1)}, 
\tag{6.6} \label{6.6}
\end{equation}
We have
\begin{displaymath}
n \geq 2\sqrt{m} \left( 1 + \frac{4 \log \log m}{\log m}\right) \geq 2\sqrt{m} \times \frac{\log (\sqrt{m}\log m/\log \log m) - 1}{\log(\sqrt{m}\log \log 
m/\log m) + 1}.
\end{displaymath}
Applying this to \eqref{6.4}, we see that the inequality $\pi(m+n) \leq \pi(m) + \pi(n)$ holds.

\textit{Case} 3. $2 \sqrt{m} \leq n \leq 2 \sqrt{m} (1 + 4 \log \log m/\log m)$. \newline
Let $r(x) = 1 + 4\log \log x/\log x$. In this latter case, a simple calculation shows that
\begin{displaymath}
n \geq 2\sqrt{m} \geq 2\sqrt{m} \times \frac{\log(2\sqrt{m}r(m)) - 1}{\log(\sqrt{m}/(2r(m))) + 1} \geq 2 \sqrt{m} \times \frac{\log n - 1}{\log(m/n) + 1}.
\end{displaymath}
Now we apply this to \eqref{6.4} to get the required inequality.

\textit{Case} 4. $2 \sqrt{m}(1-2c_1/(\log m + c_1)) \leq n \leq 2 \sqrt{m}$. \newline
In this latter case, we have
\begin{displaymath}
n \geq 2\sqrt{m} \left( 1 - \frac{2c_1}{\log m + c_1} \right) = 2\sqrt{m} \times \frac{\log(2\sqrt{m}) - 1}{\log(\sqrt{m}/2) + 1} \geq 2 \sqrt{m} 
\times \frac{\log n - 1}{\log(m/n) + 1}.
\end{displaymath}
Finally, we apply this to \eqref{6.4} to arrive at the end of the proof.
\end{proof}

\section{A Proof of Theorem \ref{thm105}}

Under the assumption that the Riemann hypothesis is true, Schoenfeld \cite[Corollary 1]{schoenfeld} showed that
\begin{displaymath}
|\pi(x) - \text{li}(x)| \leq \frac{\sqrt{x}}{8 \pi} \, \log x
\end{displaymath}
for every $x \geq 2657$. In 2018, Dusart \cite[Proposition 2.6]{dusart2018} found the following refinement.

\begin{prop}[Dusart] \label{prop701}
If the Riemann hypothesis is true, then
\begin{displaymath}
|\pi(x) - \emph{li}(x)| \leq \frac{\sqrt{x}}{8 \pi} \, \log \left( \frac{x}{\log x} \right)
\end{displaymath}
for every real $x \geq 5639$.
\end{prop}

We use Proposition \ref{prop701} to find the following proof of Theorem \ref{thm105}.

\begin{proof}[Proof of Theorem \ref{thm105}]
If $m \leq 5 \times 10^{19}$, then $m + n \leq 2m \leq 10^{20}$ and the result follows directly from Theorem \ref{thm104}. So it suffices to consider the case 
where $m \geq 5\times 10^{19}$. In order to prove the theorem, we consider the following three cases.

\textit{Case} 1. $n \geq c_2\sqrt{m}\log^3 m$. \newline
If $n \geq c_0m/\log^2m$, where $c_0$ is given as in Theorem \ref{thm103}, the result follows directly from Theorem \ref{thm103}. Hence we can assume that 
$c_2\sqrt{m}\log^3 m \leq n \leq c_0m/\log^2 m$. By Proposition \ref{prop701}, we have $\pi(m+n) \leq \text{li}(m+n) + f(m+n)$, where $f(t) = (1/(8\pi)) 
\sqrt{t} \log(t/\log t)$. Now we use the mean value theorem to get
\begin{displaymath}
\pi(m+n) \leq \text{li}(m) + \frac{n}{\log m} + f(m) + \frac{n}{16\pi\sqrt{m}} \left( \log \left( \frac{m}{\log m} \right) + 2 \right).
\end{displaymath}
Next we apply Proposition \ref{prop701} to obtain the inequality
\begin{displaymath}
\pi(m+n) \leq \pi(m) + \frac{n}{\log m} + 2f(m) + \frac{n}{16\pi\sqrt{m}} \left( \log \left( \frac{m}{\log m} \right) + 2 \right),
\end{displaymath}
which is equivalent to
\begin{displaymath}
\pi(m+n) \leq \pi(m) + 2f(m) + \frac{n}{16\pi\sqrt{m}} \left( \log \left( \frac{m}{\log m} \right) + 2 \right) + \frac{n}{\log n - 1} - n \,\frac{\log(m/n) + 
1}{\log m (\log n - 1)}.
\end{displaymath}
Since $m \geq 5 \times 10^{19}$, we have $n \geq c_2 \sqrt{m} \log^3m \geq 52\,511\,298\,895\,885$. So we can apply the inequality including $\pi(x)$ given in 
\cite[p.\:55]{dusart} to the last inequality and get
\begin{equation}
\pi(m+n) \leq \pi(m) + \pi(n) + 2f(m) + \frac{n}{16\pi\sqrt{m}} \left( \log \left( \frac{m}{\log m} \right) + 2 \right) - n \, \frac{\log(m/n) + 1}{\log m 
(\log n - 1)}. \tag{7.1} \label{7.1}
\end{equation}
Since $c_2\sqrt{m}\log^3 m \leq n \leq c_0m/\log^2m$, we see that
\begin{displaymath}
\pi(m+n) \leq \pi(m) + \pi(n) + 2f(m) + \frac{c_0\sqrt{m}}{16\pi\log^2 m} \left( \log \left( \frac{m}{\log m} \right) + 2 \right) - n \, 
\frac{2 \log \log m - c_3}{\log m (\log m  - 2\log \log m + c_3)},
\end{displaymath}
where $c_3 = \log(c_0) - 1 = -1.17409\ldots$. Now we substitute the definition of $f(t)$ to get $\pi(m+n) \leq \pi(m) + \pi(n) + g(m,n)$, where
\begin{displaymath}
g(m,n) = \frac{\sqrt{m}}{4\pi} \left( \left( 1 + \frac{c_0}{4\log^2m} \right) \log \left( \frac{m}{\log m} \right) + \frac{c_0}{2\log^2 m} \right) - n \, 
\frac{2 \log \log m - c_3}{\log m (\log m  - 2\log \log m + c_3)}.
\end{displaymath}
Clearly, it suffices to show that $g(m,n) \leq 0$. This inequality is equivalent to
\begin{equation}
n \geq \frac{\sqrt{m}}{4\pi} \left( \left( 1 + \frac{c_0}{4\log^2m} \right) \log \left( \frac{m}{\log m} \right) + \frac{c_0}{2\log^2 m} \right) 
\frac{\log m (\log m  - 2\log \log m + c_3)}{2 \log \log m - c_3}. \tag{7.2} \label{7.2}
\end{equation}
Since
\begin{displaymath}
- \log \log m + \frac{c_0}{4\log^2m} \log \left( \frac{m}{\log m} \right) + \frac{c_0}{2\log^2 m} \leq 0,
\end{displaymath}
the inequality $n \geq c_2 \sqrt{m} \log^3 m$ implies \eqref{7.2} and we get $\pi(m+n) \leq \pi(m) + \pi(n)$.

\textit{Case} 2. $c_2 \sqrt{m}\log m \log(m\log^{13}m) \leq n \leq c_2 \sqrt{m} \log^3 m$. \newline
From \eqref{7.1}, it follows that
the inequality $\pi(m+n) \leq \pi(m) + \pi(n)$ holds if
\begin{equation}
n \geq \left( c_2 \sqrt{m} \log \left( \frac{m}{\log m} \right) + \frac{c_2}{16\pi} \left( \log \left( \frac{m}{\log m} \right) + 2 \right)\log^3 m \right) 
\frac{\log m (\log(c_2\sqrt{m}\log^3m) - 1)}{\log(\sqrt{m}/(c_2\log^3 m)) + 1}. \tag{7.3} \label{7.3}
\end{equation}
We have
\begin{displaymath}
\frac{c_2}{16\pi} \left( \log \left( \frac{m}{\log m} \right) + 2 \right)\log^3 m \leq c_2 \sqrt{m} \log \log m
\end{displaymath}
and
\begin{displaymath}
\frac{\log(c_2\sqrt{m}\log^3m) - 1}{\log(\sqrt{m}/(c_2\log^3 m)) + 1} \leq 1 + \frac{13\log \log m}{\log m}.
\end{displaymath}
So if $n$ fulfills the inequality $n \geq c_2 \sqrt{m}\log m \log(m\log^{13}m)$, we get the inequality \eqref{7.3}. Hence we have $\pi(m+n) \leq 
\pi(m) + \pi(n)$.

\textit{Case} 3. $c_2 \sqrt{m}\log m \log(m\log^8m) \leq n \leq c_2 \sqrt{m}\log m \log(m\log^{13}m)$. \newline
We use \eqref{7.1} to see that
the inequality $\pi(m+n) \leq \pi(m) + \pi(n)$ holds if
\begin{equation}
n \geq \left( c_2 \sqrt{m} \log \left( \frac{m}{\log m} \right) + \frac{c_2}{16\pi} \left( \log \left( \frac{m}{\log m} \right) + 2 \right)\log m 
\log(m\log^{13}m) \right) h(m) \log m, \tag{7.4} \label{7.4}
\end{equation}
where
\begin{displaymath}
h(m) = \frac{\log(c_2\sqrt{m}\log m \log(m\log^{13}m)) - 1}{\log(\sqrt{m}/(c_2\log m \log(m\log^{13}m))) + 1}.
\end{displaymath}
Note that
\begin{displaymath}
\frac{c_2}{16\pi} \left( \log \left( \frac{m}{\log m} \right) + 2 \right)\log m \log(m\log^{13}m) \leq c_2 \sqrt{m} \log \log m
\end{displaymath}
and $h(m) \leq 1 + 8\log \log m/\log m$. So the inequality $n \geq c_2 \sqrt{m}\log m \log(m\log^8m)$ implies the inequality \eqref{7.4} and we arrive at 
the end of the proof.
\end{proof}

%
%

\section{Appendix: The Incompatibility of the HLC and the Prime $k$-tuples Conjecture}

To formulate the Prime $k$-tuples Conjecture, we first introduce the following definition.

\begin{defi}
A $k$-tuple of distinct integers $b_1, \ldots, b_k$ is \textit{admissible} if for each prime $p$, there is some congruence class mod $p$ which contains none of 
the $b_i$.  
\end{defi}

\begin{prim}
Let $b_1, \ldots, b_k$ be an admissible $k$-tuple of integers. Then there exist infinitely many positive integers $n$ for which all of the values $n + b_1, 
\ldots, n + b_k$ are prime.
\end{prim}

\begin{rema}
The Prime $k$-tuples Conjecture is a special case of Schinzel's Hypothesis H \cite[p.\:188]{schinzel}.
\end{rema}

In order to show that the HLC and the Prime $k$-tuples Conjecture are incompatible, Hensley and Richards \cite{hensley} used the following function which 
was introduced by Schinzel and Sierpi\'nski \cite[p.\:201]{schinzel}.

\begin{defi}
Let the function $\rho^{\ast}: \N \to \N$ be defined by
\begin{displaymath}
\rho^{\ast}(m) = \max_{n \in \N} |\{ k \in \N \mid n < k \leq m+n, \text{gcd}(k, m!) = 1 \}|. 
\end{displaymath}
This function describes the maximum number of positive integers in each interval $(n,m+n]$ that are relatively prime to all positive integers less than or 
equal to $m$.
\end{defi}

Under the assumption that the Prime $k$-tuples Conjecture is true, Schinzel and Sierpi\'nski \cite[pp.\: 204--205]{schinzel} found the identity
\begin{equation}
\rho^{\ast}(m) = \limsup_{n \to \infty} (\pi(m+n) - \pi(n)). \tag{8.1} \label{8.1}
\end{equation}
Hensley and Richards \cite[p.\:380]{hensley} proved that for every real number $\e$ there exists a $m_0(\e)$ so that
\begin{displaymath}
\rho^{\ast}(m) - \pi(m) \geq (\log 2 - \e) \times \frac{m}{\log^2 m}
\end{displaymath}
for every $m \geq m_0(\e)$. In particular, the last inequality gives
\begin{equation}
\lim_{m \to \infty} (\rho^{\ast}(m) - \pi(m)) = \infty. \tag{8.2} \label{8.2}
\end{equation}
So if the Prime $k$-tuples Conjecture is true, we can combine \eqref{8.1} and \eqref{8.2} to see that for every sufficiently large values of $m$ there exist 
infinitly many positive integers $n$ so that the inequality
\begin{displaymath}
\pi(m+n) > \pi(m) + \pi(n)
\end{displaymath}
holds which contadicts the HLC.

\section*{Acknowledgement}

I would like to express my great appreciation to Thomas Leßmann for writing a C++ program in order to verify Proposition \ref{prop204}. I would also like to 
thank Pierre Dusart whose PhD thesis inspired me to deal with this subject. Furthermore, I thank R. for being a never-ending inspiration.


\begin{thebibliography}{10}
\bibitem{axler16} C. Axler, New bounds for the prime counting function, {\it Integers} {\bf 16} (2016), Paper No. A22, 15 pp.
\bibitem{axler2017} C. Axler, New Estimates for Some Functions Defined Over Primes, {\it Integers} \textbf{18} (2018), Paper No. A52, 21 pp.
\bibitem{dusart} P. Dusart, In\'egalit\'es explicites pour $\psi(X)$, $\theta(X)$, $\pi(X)$ et les nombres premiers, {\it C. R. Math. Acad. Sci. Soc. 
R. Can.} {\bf 21} (1999), no. 2, 53--59.
\bibitem{dusart2} P. Dusart, Sur la conjecture $\pi(x+y) \leq \pi(x) + \pi(y)$, {\it Acta Arith.} \textbf{102} (2002), no. 4, 
295--308.
\bibitem{dusart2018} P. Dusart, Estimates of the $k$th prime under the Riemann hypothesis, {\it Ramanujan J.} \textbf{47} (2018), no. 1, 141--154.
\bibitem{erdos} P. Erdös, Some unsolved problems, {\it Michigan Math. J.} \textbf{4} (1957), 291--300.
\bibitem{rodemich} D. M. Gordon and G. Rodemich, Dense admissible sets, {\it Algorithmic number theory} (Portland, OR, 1998), 216--225, Lecture Notes in 
Comput. Sci., 1423, Springer, Berlin, 1998.
\bibitem{hardy} G. H. Hardy and J. E. Littlewood, Some problems of `Partitio numerorum'; III: On the expression of a number as a sum of primes, {\it Acta 
Math.} \textbf{44} (1923), no. 1, 1--70.
\bibitem{hensley} D. Hensley and I. Richards, On the incompatibility of two conjectures concerning primes, {\it Analytic number theory} (Proc. Sympos. 
Pure Math., Vol. XXIV, St. Louis Univ., St. Louis, Mo., 1972), 123--127, Amer. Math. Soc., Providence, R.I., 1973.
\bibitem{kopetzky} H. G. Kopetzky and W. Schwarz, Two conjectures of B. R. Santos concerning totitives, {\it Math. Comp.} \textbf{33} (1979), no. 146, 841--844.
\bibitem{landau} E. Landau, {\it Handbuch der Lehre von der Verteilung der Primzahlen}, Teubner, 1909.
\bibitem{lionnet} F. J. E. Lionnet, Question 1075, {\it Nouv. Ann. Math.} \textbf{11} (1872), p. 190.
\bibitem{pan3} L. Panaitopol, A formula for $\pi(x)$ applied to a result of Koninck-Ivi\'{c}, \textit{Nieuw Arch. Wiskd}. (5) \textbf{1} (2000),
no. 1, 55--56.
\bibitem{pana1} L. Panaitopol, Checking the Hardy-Littlewood conjecture in special cases, {\it Rev. Roumaine Math. Pures Appl.} \textbf{46} (2001), no. 4, 
465--470.
\bibitem{pana2} L. Panaitopol, A special case of the Hardy-Littlewood conjecture, {\it Math. Rep. (Bucur.)} \textbf{4}(54) (2002), no. 3, 265--268.
\bibitem{rosser1975} J. B. Rosser and L. Schoenfeld, Sharper bounds for the Chebyshev functions $\theta(x)$ and $\psi(x)$, Collection of articles dedicated to 
Derrick Henry Lehmer on the occasion of his seventieth birthday, \textit{Math. Comp.} \textbf{29} (1975), 243--269.
\bibitem{schinzel2} A. Schinzel, Remarks on the paper ``Sur certaines hypothèses concernant les nombres premiers'', {\it Acta Arith.} \textbf{7}, (1961/1962), 
1--8.
\bibitem{schinzel} A. Schinzel and W. Sierpi\'nski, Sur certaines hypothèses concernant les nombres premiers, {\it Acta Arith.} \textbf{4} (1958), 185--208.
\bibitem{schoenfeld} L. Schoenfeld, Sharper bounds for the Chebyshev functions $\theta(x)$ and $\psi(x)$. II, {\it Math. Comp.} \textbf{30} (1976), no. 134, 
337--360.
\bibitem{segal} S. Segal, On $\pi(x+y) \leq \pi(x) + \pi(y)$, {\it Trans. Amer. Math. Soc.} \textbf{104} (1962), 523--527.
\bibitem{udrescu} V. St. Udrescu, Some remarks concerning the conjecture $\pi(x+y) \leq \pi(x) + \pi(y)$, {\it Rev. Roumaine Math. Pures Appl.} \textbf{20} 
(1975), no. 10, 1201--1209.
\end{thebibliography}
\end{document}